\documentclass{artjlt_modified}
\usepackage{amsmath,amsfonts,amssymb, hyperref,xypic}
\title{Orbits of the Centralizer of a Linear Operator} 
\author{Paul Best, Marco Gualtieri, Patrick Hayden} 
\lastname{Best, Gualtieri, Hayden} 



\address{Paul Best\\               
Warburg Pincus LLC\\
London, UK\\
}
\address{Marco Gualtieri\\               
Department of Mathematics\\
University of Toronto \\
Toronto, Canada\\
mgualt@math.toronto.edu
}
\address{Patrick Hayden\\               
School of Computer Science\\
McGill University\\
Montreal, Canada\\
patrick@cs.mcgill.ca
}

\begin{document}
\newcommand{\nc}{\newcommand}
\nc{\End}{\operatorname{End}}
\nc{\Hom}{\operatorname{Hom}}
\nc{\Aut}{\operatorname{Aut}}
\nc{\GL}{\operatorname{GL}}
\nc{\Ker}{\operatorname{Ker}}
\nc{\Orb}{\operatorname{Orb}}
\nc{\Span}{\operatorname{Span}}
\renewcommand{\Im}{\operatorname{Im}}
\nc{\aUa}{\uparrow}\nc{\aRa}{\rightarrow}\nc{\aDa}{\nearrow}
\nc{\rar}{\rightarrow}
\nc{\lrar}{\longrightarrow}
\nc{\aUDa}{\updownarrow}
\nc{\DT}{\mathcal{D}(T)}
\nc{\OO}{\mathcal{O}}
\nc{\bul}{\bullet}
\nc{\CC}{\mathbb C}
\nc{\Cc}{\mathcal C}
\nc{\Ss}{\mathcal S}
\nc{\KK}{\mathbb K}
\nc{\RR}{\mathbb R}
\nc{\PP}{\mathcal P}
\nc{\NN}{\mathbb N}
\nc{\ZZ}{\mathbb Z}
\renewcommand{\vec}[1]{\mathbf{#1}}
\nc{\vv}{\mathbf{v}}
\nc{\zz}{\mathbf{z}}
\nc{\caut}[1]{C(#1)}
\nc{\cend}[1]{\mathfrak{c}(#1)}
\nc{\dd}[1]{D(#1)}
\nc{\BAR}[1]{\overline{#1}}
\nc{\ann}{\mathrm{ann}}


\maketitle

\begin{abstract}
    We describe the orbit structure for the action of the centralizer group $\caut{T}$ of a linear
    operator $T$ on a finite-dimensional complex vector space $V$.  The main application is to the
    classification of solutions to a system of first-order ODEs with constant coefficients.  We
    completely describe the lattice structure on the set of orbits and provide a generating
    function for the number of orbits in each dimension.
\end{abstract}

\section{Introduction} \label{sintro}
Let $V$ be a finite-dimensional complex vector space, and fix $T\in \End(V)$.
Consider the system of linear ordinary differential equations with constant
coefficients
    \begin{equation} \label{eode}
    \vec{x}' = T\vec{x}.
    \end{equation}
Let $\mathcal{S}$ denote the set of solutions $\vec{x}:\CC\rightarrow V$ of Equation~\ref{eode}.  The centralizer group of the operator $T$, given by
\[
C(T) = \left\{U\in \GL(V)\ :\ UT=TU\right\},
\] 
may also be characterized as the group of invertible operators $U\in GL(V)$ such that
$U\circ\vec{x}\in\mathcal{S}$ for each $\vec{x}\in\mathcal{S}$. In this way, $\caut{T}$ acts on
$\mathcal{S}$, and we may consider two solutions to be equivalent when they are in the same $\caut{T}$-orbit.

The evaluation map $\vec{x}\mapsto \vec{x}(0)$ defines a bijection $\mathcal{S}\rightarrow V$ with inverse $\vec{x}_0\mapsto (t\mapsto \exp(tT)\vec{x}_0)$, which intertwines the natural $\caut{T}$-actions on $\mathcal{S}$ and $V$. Therefore, equivalence classes of solutions in $\mathcal{S}$ are in one-to-one correspondence with $\caut{T}$-orbits in $V$.  In short, to classify solutions to Equation \ref{eode}, we must describe the orbit structure of $V$ under the action of $\caut{T}$.

\section{Finitely many orbits}

Consider an operator with only one Jordan block, i.e. $T = \lambda I + N$, where $N$ is nilpotent of
degree $n=\dim V$.  In this case, the only operators which commute with $T$ are the polynomials in
$T$. The centralizer may be described explicitly as follows:
\begin{eqnarray*}
    \caut{T} = \left\{\sum_{i=0}^{n-1} a_i N^i\ \ :\ \ a_i\in\CC, a_0\neq 0\right\}.
\end{eqnarray*}
As a result, the orbits of $\caut{T}$ on $V$ are precisely given by $\OO_i = F_i - F_{i-1}$, where
$F_{-1}=\varnothing$ and
\[
F_i = \ker N^i,\ \ i\geq 0,
\]
defines the full flag of $T$-invariant subspaces associated to the nilpotent operator $N$.
\begin{lemma}\label{flag}
    If $T$ has only one Jordan block, then there are exactly $\dim V + 1$ orbits
    $\OO_0,\ldots, \OO_n$ of $\caut{T}$ on $V$, corresponding to the full flag of invariant
    subspaces $F_0\subset\cdots\subset F_n=V$ via $\OO_i = F_i-F_{i-1}$.
\end{lemma}

In the general case, $V$ admits a decomposition $V=\oplus_i V_i$ such that $T|_{V_i}$ has a single
Jordan block, and the centralizer of $T$ is more complicated, as we describe in
Section~\ref{thecentralizer}. However, the product of the centralizers
$\dd{T}= \oplus_i\caut{T|_{V_i}}$ is \emph{contained} in $\caut{T}$.  There are only finitely many
orbits of $\dd{T}$, since they are products of $\caut{T|_{V_i}}$-orbits.  The orbits of $\dd{T}$,
however, are refinements of the orbits of the larger group $\caut{T}$, hence there can only be
finitely many orbits of the centralizer group.
\begin{theorem}\label{finitt}
    There are finitely many orbits of $\caut{T}$ in $V$.
\end{theorem}

Let $\cend{T}$ be the algebra of operators commuting with $T$.  It contains the centralizer group
$\caut{T}$ as an open dense subset, and may be identified with the Lie algebra of $\caut{T}$.  It
follows that each orbit of $\caut{T}$ in $V$ is an open dense subset of a $\cend{T}$-invariant
subspace of $V$.  We now show, using the finiteness result above, that $\caut{T}$-orbits are in
one-to-one correspondence with $\cend{T}$-invariant subspaces.

\begin{theorem}
    Orbit closure is a bijection from the set of orbits of $\caut{T}$ to the set of
    $\cend{T}$-invariant subspaces of $V$.
\end{theorem}
\begin{proof}
    We show the map $\caut{T}v\mapsto\BAR{\caut{T}v} = \cend{T}v$ is a bijection by providing its
    inverse.  If $Y\subset V$ is $\cend{T}$-invariant, let $\OO_Y$ be the complement in $Y$ of the
    union of its $\cend{T}$-invariant proper subspaces.  Theorem~\ref{finitt} ensures there are only
    finitely many such subspaces, hence $\OO_Y$ is nonempty. Furthermore, $\OO_Y$ must be a union of
    orbits of $\caut{T}$, but it cannot contain more than one orbit, since $Y$ cannot contain two
    disjoint open dense sets.  Hence the map $Y\mapsto \OO_Y$ is the required inverse.
\end{proof}

In view of the above bijection, we proceed to classify the $\caut{T}$ orbits by completely
describing the invariant subspaces for the action of the algebra $\cend{T}$ on $V$.

\section{The centralizer algebra of a linear operator}\label{thecentralizer}

To identify the $\cend{T}$-invariant subspaces of $V$, we need a convenient description of the
algebra $\cend{T}$ itself. View the vector space $V$ as a $\CC[x]$-module, where $x\vv=T(\vv)$ for
$\vv\in V$.  This point of view is particularly useful for us, because of the following.
\begin{proposition}
    A linear operator $U$ commutes with $T$ if and only if it is a $\CC[x]$-module
    endomorphism $V\rightarrow V$.  In other words, $\cend{T} = \End_{\CC[x]}(V)$.           
\end{proposition}
Let the minimal polynomial of $T$ be $\prod_{\lambda} p_\lambda^{k_\lambda}$, where $p_\lambda = (x-\lambda)$ and the product is over  distinct
eigenvalues $\lambda\in\mathrm{Spec}(T)$. The associated generalized eigenspace decomposition is 
\[
V = \bigoplus_{\lambda\in\mathrm{Spec}(T)} V_\lambda,
\]  
with $V_\lambda=\ker (T-\lambda)^{k_\lambda}$.  A priori, the endomorphism algebra decomposes as a direct sum of the
components $\Hom_{\CC[x]}(V_\lambda,V_{\lambda'})$, but for $\lambda\neq \lambda'$ this is the zero vector space, since a
morphism $\phi:V_{\lambda}\rightarrow V_{\lambda'}$ satisfies $0=\phi(p_\lambda^{k_\lambda} \vv) = p_\lambda^{k_\lambda}\phi(\vv)$, and
$p_\lambda$ is invertible on $V_{\lambda'}$ for $\lambda\neq \lambda'$.  Hence we obtain the following decomposition of
$\cend{T}$: 
\begin{proposition}
    The centralizer algebra $\cend{T}$ decomposes as a direct sum of centralizers of the
    restrictions $T_\lambda$ of $T$ to the generalized eigenspaces $V_\lambda = \ker(T-\lambda)^{k_\lambda}$.
\end{proposition}
Consequently, orbits of the full centralizer algebra are products of orbits of the summands
$\cend{T_\lambda}$, and we need only consider the case of a single eigenvalue. 
So, consider the case where $T\in\End(V)$ has minimal polynomial $(x-\lambda)^k$, and choose a Jordan
decomposition of $V$, as follows:
\begin{equation}\label{singeig}
V = V^1\oplus\cdots\oplus V^k, 
\end{equation}
where each $V^i = V^i_1\oplus\cdots\oplus V^i_{m_i}$ is a sum of $m_i$ cyclic modules with
annihilator $(x-\lambda)^i$, and we take $V^i=0$ when $m_i=0$.  In other words, $T|_{V^i}$ consists of
$m_i$ repeated Jordan blocks of size $i$.  We now compute the module homomorphisms between
individual summands of $V^i$ and $V^j$.
\begin{proposition}\label{homs} Let $M_i$ be the cyclic module $\CC[x]/p^i$ for $p=(x-\lambda)$, $\lambda\in\CC$.  Then 
    \[
    \Hom_{\CC[x]}(M_i,M_{i'})=
    \begin{cases}
        M_{i'} & \text{ for }i\geq i', \\
        p^{i'-i}M_{i'} & \text{ for }i\leq i'.
    \end{cases}    
    \]
    \end{proposition}
\begin{proof}
    Since $M_i,M_{i'}$ are cyclic, $\phi\in \Hom(M_i,M_{i'})$ is determined by $[1]\mapsto f$ for
    $f\in M_{i'}$ such that $p^if = 0$.  For $i\geq i'$ this does not impose a condition on $f$, but for
    $i'>i$ we obtain $f\in p^{i'-i}M_{i'}$, as required.
\end{proof}

\begin{example}\label{theex}
    Suppose $V$ decomposes as \(V^2\oplus V^3=\CC[x]\vv_2\oplus\CC[x]\vv_3,\) where
    $\ann(\vv_2)=(x^2)$ and $\ann(\vv_3)=(x^3)$.  Then
    $(\vv_2, x\vv_2, \vv_3,x\vv_3,x^2\vv_3)$ is a Jordan basis in which $T$ has
    the following Jordan form: 
\[T=\small\left[\begin{array}{ccccccc}
            0& & & & \\
            1&0 & & & \\
            &  &0 & & \\
            & & 1&0&  \\
            & & &1&0 
                \end{array}\right]\normalsize
    \]
    $\phi\in\cend{T}$ then decomposes as $\phi_{22}+\phi_{23}+\phi_{32} + \phi_{33}$, where
    $\phi_{ij}\in\Hom(V^i,V^j)$.  By Proposition~\ref{homs}, we have
    $\phi_{22}(\vv_2)=(a+bx)\vv_2$, $\phi_{33}(\vv_3)=(c+dx + ex^3)\vv_3$,
    $\phi_{23}(\vv_2)=(hx+kx^2)\vv_3$, and $\phi_{32}(\vv_3)=(f+gx)\vv_2$, where $a,b,c,d,e,f,g,h,k$ are
    arbitrary complex numbers. Writing $\phi$ in terms of the Jordan basis, we obtain:
    \[\small\cend{T} = \left\{\left[\begin{array}{ccccccc}
            a& & f& &  \\
            b&a&g &f&  \\
             & &c&& \\
            h& &d&c & \\
            k&h&e&d&c 
       \end{array}\right]\ :\ a,b,c,d,e,f,g,h,k\in\CC\right\}\normalsize
    \]    
\end{example}

\section{Classification of $\cend{T}$-invariant subspaces}\label{sinvariantsubspaces}

For a single cyclic module $M_i = \CC[x]/p^i$, Lemma~\ref{flag} shows that there are $i+1$ invariant
subspaces for the action of $\cend{T}$, forming a full flag $F_0\subset \cdots\subset F_i=M_i$.  We
may write $F_l = p^{i-l}M_i$.  We now show that any $\cend{T}$-invariant subspace in the sum of cyclic modules~\eqref{singeig} decomposes into a direct sum of its projections to the cyclic summands.
\begin{theorem}\label{scheme}
Let $T\in\End(V)$ have minimal polynomial $p^k$ for $p=(x-\lambda)$, and let $m_i$ be the number of Jordan blocks of size $i$, so that we may choose a Jordan decomposition $V =
\oplus_{i=1}^k V^i$, where $V^i = V^i_1\oplus\cdots\oplus V^i_{m_i}$ is a sum of cyclic
modules isomorphic to $\CC[x]/p^i$ (and we set $V^i=\{0\}$ for $m_i=0$).  Then $W\subset V$ is a  $\cend{T}$-invariant subspace if and
only if the following three conditions hold:
\begin{enumerate}
\item\label{one} $W$ is a direct sum of subspaces of the form $p^{i-l}V^i_j$.
\item\label{two} If $p^{i-l}V^i_j\subset W$, then $p^{i'-l}V^{i'}_{j'}\subset W$ for all $i'\geq i$ and all $j'$. 
\item\label{three} If $p^{i-l}V^i_j\subset W$, then $p^{i-l}V^{i'}_{j'}\subset W$ for all $i'\leq i$ and all $j'$.
\end{enumerate} 
\end{theorem}
\begin{proof}
    The projection $\pi^i_j$ from $V$ to each cyclic summand $V^i_j$ commutes with $T$; therefore
    $\pi^i_j\in\cend{T}$.  So, if $W\subset V$ is $\cend{T}$-invariant, it must contain all of its
    projections onto the cyclic summands, and we obtain $W = \oplus_{i,j} \pi^i_j W$.  Moreover,
    each of $\pi^i_j W$ is $\cend{T|_{V^i_j}}$-invariant and hence must coincide with some member
    $p^{i-l}V^i_j$ of the flag, proving part~\ref{one}.

    $W$ is $\cend{T}$-invariant if and only if $\cend{T}p^{i-l}V^i_j\subset W$ for all summands
    $p^{i-l}V^i_j$ present in $W$.  Recall that each element in $\cend{T}$ is a sum of morphisms
    $\phi\in\Hom(V^i_j,V^{i'}_{j'})$.  By Proposition~\ref{homs}, we see that the action map
\[
\Hom(V^i_j,V^{i'}_{j'})\otimes p^{i-l}V^i_j\longrightarrow V^{i'}_{j'}
\]
is surjective onto $p^{i'-i} p^{i-l}V^{i'}_{j'} = p^{i'-l}V^{i'}_{j'}$ for $i'\geq i$ and any $j'$.
It is also onto
$p^{i-l}V^{i'}_{j'}$ for $i'\leq i$ and any $j'$, as required. 
\end{proof}
Theorem~\ref{scheme} has a helpful interpretation as defining a poset, as we now describe.  First note that
if $p^{i-l} V^i_j$ is contained in an invariant subspace $W$, then $p^{i-l}V^i_{j'}$ must also be
contained for all $j' = 1,\ldots, m_i$.  Hence we treat the direct sum $\oplus_{j} p^{i-l} V^i_j$ as a
single subspace, which we denote by $m_i p^{i-l} V^i$.  We define a partial order on the set $\PP=\{m_ip^{i-l}V^i\}$ of these subspaces
by setting $A\leq B$ when $\cend{T}(B)$ contains $A$.  By Theorem~\ref{scheme}, the Hasse diagram of $\PP$
is as drawn in Figure~\ref{graph}, in the (fictitious) situation that all multiplicities $m_i$ are nonzero.
\begin{figure}[h]
\begin{equation*}
\xymatrix@=1.6em@ur{
m_1 V^1 \ar@{-}[r]\ar@{-}[d]& m_2 V^2 \ar@{-}[r]\ar@{-}[d] & m_3 V^3 \ar@{-}[r]\ar@{-}[d] & m_4 V^4
\ar@{-}[r]\ar@{-}[d] & \cdots\\
m_2p V^2 \ar@{-}[r]\ar@{-}[d] & m_3p V^3 \ar@{-}[r]\ar@{-}[d] & m_4 pV^4\ar@{-}[r]\ar@{-}[d] & \cdots &\\
 m_3p^2 V^3 \ar@{-}[r]\ar@{-}[d] & m_4 p^2V^4\ar@{-}[r]\ar@{-}[d] & \cdots &&\\
m_4 p^3V^4\ar@{-}[r]\ar@{-}[d] & \cdots &&&\\
\cdots& &&&\\
}
\end{equation*}
\caption{Poset $\PP$ describing the action of $\cend{T}$ on $m_ip^{i-l}V^i$}\label{graph}\normalsize
\end{figure}
This poset appears in the study of representations of $\mathfrak{gl}_n(\CC)$, where it is known as the Gelfand-Tsetlin poset~\cite{GT}.
\begin{corollary}\label{outdegree}
$W\subset V$ is a $\cend{T}$--invariant subspace if and only if it is a direct sum of subspaces
$m_ip^{i-l} V^i$ which form a decreasing subset\footnote{We say $I\subset \PP$ is decreasing if $x\in I$ and $y\leq x$ imply that $y\in I$.} in the above poset $\PP$.
\end{corollary}
Of course, the linear operator $T$ has Jordan blocks of only a finite number of possible sizes.  Hence, all but
a finite number of the multiplicities $m_i$ are zero, and so the corresponding vertices in the poset
$\PP$ do not contribute to any $\cend{T}$--invariant subspaces of which they are summands.
As a result, the $\cend{T}$--invariant subspaces are in bijection with the decreasing subsets of a
\emph{subposet} of $\PP$, defined by the vertices with nonzero multiplicities $m_i$.  

Furthermore,
$\cend{T}$--invariant subspaces form a lattice, under the usual operations of sum and intersection
of subspaces.  This lattice structure clearly coincides with the usual lattice structure on
decreasing subsets of the poset $\PP$.  Summarizing, we obtain the following classification.
\begin{theorem}\label{classi}
Let $T\in\End(V)$ have a single eigenvalue and Jordan blocks whose sizes define a finite subset $B\subset
\NN$.   The lattice of $\cend{T}$--invariant subspaces of $V$ is isomorphic to the the lattice of decreasing subsets in
$\PP_B$,
the subposet of $\PP$ generated by the columns of length $i\in B$.
\end{theorem}
\begin{example}
If $T$ is nilpotent, with any number of Jordan blocks, but of sizes 1, 3, and 5 only, then the $\cend{T}$--invariant
subspaces are in bijection with decreasing subsets of the following subposet of the Gelfand-Tsetlin
poset:
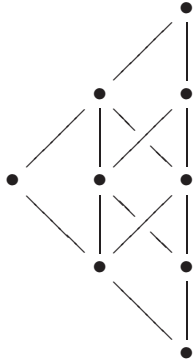
\begin{equation*}\small
\xymatrix@=.7em@ur{
\bullet \ar@{-}[rr]\ar@{-}[dd]& & \bullet \ar@{-}[rr]\ar@{-}[dl]\ar@{-}[dd] |!{[dl];[dr]}\hole & & \bullet \ar@{-}[dl]\\
& \bullet \ar@{-}[rr]\ar@{-}[dl]\ar@{-}[dd]|!{[dl];[dr]}\hole & & \bullet \ar@{-}[dl] &\\
 \bullet \ar@{-}[rr]\ar@{-}[dd] & & \bullet \ar@{-}[dl] &&\\
& \bullet \ar@{-}[dl] &&&\\
\bullet& &&&\\
}
\end{equation*}
\end{example}
\begin{remark}
It is
well-known~\cite{Stanley} that the decreasing subsets of a poset form a distributive
lattice, which is self-dual when the original poset is.  As a result, we may conclude that the
lattice of $\cend{T}$--invariant subspaces is a self-dual distributive lattice.
\end{remark}
\section{Orbit lattice}\label{scombinatorics}

Theorem~\ref{classi} characterizes the lattice of $\cend{T}$--invariant subspaces, and therefore the
lattice of centralizer orbits, as the lattice of decreasing subsets of a poset constructed entirely
from the knowledge of the sizes (not the multiplicities) of the Jordan blocks which occur in each
generalized eigenspace.  We now give a more explicit description of the orbit lattice, without
reference to the Gelfand-Tsetlin poset.

The orbit lattice is a Cartesian product of the orbit lattices in each generalized eigenspace
$V_\lambda$.  We first determine the lattice $\Gamma_\lambda$ corresponding to a single generalized
eigenspace, using the notation from Theorem~\ref{classi}.

Assume $T$ has a single eigenvalue and let $B\subset \NN$ be the set of sizes of Jordan blocks in
the Jordan decomposition of $T$.  For each block size $i\in B$, let $C_i$ be the corresponding
column of length $i$ in the subposet $\PP_{B}\subset \PP$.  The columns are totally ordered
$(C_{i_1}, C_{i_2}, \ldots)$ from smallest to largest, reading from left to right in the poset
$\PP_{B}$.

A decreasing subset $X\subset \PP_{B}$ is determined by the sequence
$(\#(X\cap C_{i_k}))_{k\in\NN}$, which counts the number of elements in each column.  Alternatively,
we may represent this information as a sequence $\delta^X = (\delta_1^X, \delta^X_2, \ldots)$ of
successive increments, in the following way.  Let
\begin{equation}\label{bijj}
    \delta^X_k = \begin{cases}
        \#(X\cap C_{i_1}) & k=1 \\
        \#(X\cap C_{i_k}) - \#(X\cap C_{i_{k-1}}) & k>1.
    \end{cases}
\end{equation}
The condition that $X$ be a decreasing subset is easier to state in terms of the sequence
$\delta^X$: for all $k$,
\begin{equation}\label{incbij}
    0\leq \delta^X_k \leq \Delta_k,
\end{equation}
where $\Delta_1 = i_1$ and $\Delta_k = i_k - i_{k-1}$ for $k>1$.  In other words, the intersection
of $X$ with each successive column $C_k$ must not decrease in length, and any increase is bounded by
the increment $\Delta_k$ in the total column length.
\begin{Definition}\label{sbi}
    Let $B\subset \NN$ be the set of sizes of Jordan blocks for $T$, for a fixed eigenvalue.  We
    define the sequence of block increments $\Delta = (\Delta_k)_{k\in\NN}$ as follows:
    \begin{align*}
    \Delta_1 &= i_1,\\
	\Delta_k &= i_k - i_{k-1},\ \ \text{ for } k>1,
\end{align*}
where $B=\{i_1,i_2,\ldots\}$, in increasing order so that $i_k<i_{k+1}$ for all $k$.
\end{Definition}
We may then rephrase the condition~\eqref{incbij} as follows.
\begin{proposition}
    Equation~\ref{bijj} establishes a bijection between decreasing subsets $X\subset \PP_B$ and
    elements in 
    \[
    [\Delta_1]\times [\Delta_2]\times\cdots\times[\Delta_{\#B}],
    \]
where $(\Delta_k)_{k\in\NN}$ is the sequence of block increments of $T$ and $[n]$ is the set
$\{0,1,\ldots, n\}$.
\end{proposition}
The partial order on decreasing subsets of $\PP_B$ may be described as follows: $X\leq X'$ when
$\#(X\cap C_k) \leq \#(X'\cap C_{k})$ for all $k$.  In terms of the corresponding sequences of
increments $\delta^X, \delta^{X'}$, this is simply the condition
\[
\delta^X_1 + \cdots + \delta^X_k \leq \delta^{X'}_1 + \cdots + \delta^{X'}_{k},\text{ for all } k.
\]
This partial order defines a natural poset structure on the product
$\prod_k[\Delta_k]$, for any sequence $(\Delta_k)_{k\in\NN}$ of natural numbers.
\begin{Definition}\label{lat}
    Given the sequence $\Delta=(\Delta_k)_{k\in\NN}$ of natural numbers, let $[\Delta_k] =
    \{0,\ldots, \Delta_k\}$ and  define a partial order on
    $\Gamma_\Delta = \prod_k[\Delta_k]$ as follows: for $r=(r_i)_{i\in\NN}$ and $s=(s_i)_{i\in\NN}$
    in $\Gamma_\Delta$, $r\leq s$ if and only if
\begin{equation}\label{order}
    \sum_{i\leq k} r_i\leq
    \sum_{i\leq k}s_i,\text{ for all } k\in\NN.
\end{equation}
\end{Definition}
We conclude with the explicit description of the full orbit lattice in terms of the posets defined above.
\begin{theorem}\label{latticethm}
    For each distinct eigenvalue $\lambda$ of $T\in\End(V)$, let $\Delta^\lambda$ be
    the associated sequence of block increments, as in Definition~\ref{sbi}.  Then the lattice of orbits of $C(T)$ is
    isomorphic to the Cartesian lattice product
    \[
    \prod_{\lambda\in\mathrm{Spec}(T)} \Gamma_{\Delta^\lambda},
    \]
    for $\Gamma_{\Delta^\lambda}$ as given in Definition~\ref{lat}.
\end{theorem}
\begin{example}\label{bigex}
    Let $T\in\End(V)$ be nilpotent, with Jordan blocks of sizes 1, 3, and 5 only.  The sequence of
    block increments is then $\Delta = (1, 2, 2)$, and so the $C(T)$--orbit lattice is given by $[1]\times [2]\times[2]$,
with the ordering specified by~\eqref{order}.  The
    Hasse diagram of this lattice is given below.
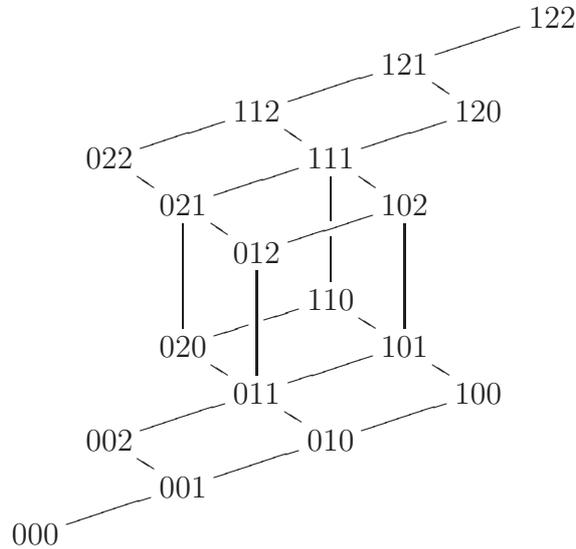
\begin{equation*}
    \xymatrix@=4pt{
        && & & & &&122\ar@{-}[lld] & \\
        && & & &121\ar@{-}[lld]\ar@{-}[rd]& & & \\
        &&&112\ar@{-}[lld]\ar@{-}[rd]  & &&120\ar@{-}[lld] & & \\
        &022 \ar@{-}[rd]& & &111\ar@{-}[lld]\ar@{-}[rd]\ar@{-}[ddd] |!{[ddl];[dr]}\hole& & & & \\
        &&021\ar@{-}[rd]\ar@{-}[ddd] &&&102\ar@{-}[lld]\ar@{-}[ddd] &    & & \\
        && &012\ar@{-}[ddd] & & &    & & \\
        && & &110\ar@{-}[lld]|!{[ul];[ddl]}\hole\ar@{-}[rd] & &    & & \\
        &&020 \ar@{-}[rd] & & &101\ar@{-}[lld]\ar@{-}[rd] &    & & \\
        && &011 \ar@{-}[lld]\ar@{-}[rd]& & & 100\ar@{-}[lld]   & & \\
        &002 \ar@{-}[rd]  && & 010\ar@{-}[lld]& &    & & \\
        && 001\ar@{-}[lld]& & & &    & & \\
        000 && & & & &    & & 
    }
\end{equation*}
\end{example}
\section{Counting orbits}
By Theorem~\ref{latticethm}, centralizer orbits are in bijection with elements in the Cartesian
product 
\[
\prod_{\lambda\in\mathrm{Spec}(T)} \prod_{k\in\NN} [\Delta^\lambda_k],
\]
where the first product is over the distinct eigenvalues and the second is over the finite number of
nonzero block increments associated to a fixed eigenvalue.  The
cardinality of $[\Delta^\lambda_k]$ is $1 + \Delta^\lambda_k$, so we obtain a simple formula for the total
number of centralizer orbits in terms of the set of Jordan block sizes in each generalized
eigenspace.

In this section, we use the theory of generating functions~\cite{Stanley} (c.f. Prop 1.4.4) to refine this count,
giving the number of centralizer orbits of dimension $n$.  Unlike the total number of orbits, this
depends on the multiplicities $m_i$ of the vertices in the Gelfand-Tsetlin poset, and hence the
multiplicity of the Jordan blocks of a fixed size in each generalized eigenspace.

First consider the case that $T$ has a single eigenvalue, let $B = (i_1, i_2, \ldots)$ be the sizes of Jordan
blocks in increasing order as before, and for each $i_k\in B$, let $m_{i_k}$ be the multiplicity of
the Jordan block of size $i_k$. Let $(C_{i_1}, C_{i_2}, \ldots)$ be the columns of the subposet
$\PP_B$ as before.   If $X\subset\PP_B$ is a decreasing subset, then the centralizer
orbit it represents has dimension given by the sum of the  $\#(X\cap C_{i_k})$, where each term is
weighted by the multiplicity $m_{i_k}$.  

As a result, the sequence of increments $\delta^X = (\delta^X_1,\delta^X_2,\ldots)$ defined
by~\eqref{bijj} can be used to compute the dimension of the orbit $\OO_X$ by the following formula:
\[
\dim\OO_X = m_{i_1}\delta^X_1 + m_{i_2}(\delta^X_1+\delta^X_2) + \cdots +
m_{i_k}(\delta^X_1+\cdots+\delta^X_k) + \cdots.
\]
From this, we define the following generating function:  let $M_n = \sum_{k\geq n} m_{i_k}$ be the
tail sums of the sequence of multiplicities, and define 
\[
f(x) = \prod_{n\in\NN} \left( \sum_{i=0}^{\Delta_n} x^{iM_n}\right).
\]
Then the coefficient of $x^m$ in this polynomial is the number of distinct centralizer orbits of
dimension $m$.    We conclude with the generating function in the case of multiple eigenvalues.
\begin{theorem}
    For each eigenvalue $\lambda$ of $T\in\End(V)$, let $(\Delta^\lambda_k)_{k\in\NN}$ be the
    associated sequence of Jordan block increments, let $(m^\lambda_{i_k})_{k\in\NN}$ be the
    sequence of multiplicities of Jordan blocks of size $i_k$, as above, and let
    $M^\lambda_n = \sum_{k\geq n} m^\lambda_{i_k}$ be the tail sums of these multiplicities.  Define
    the polynomial
\[
f_\lambda(x) =\prod_{k\in\NN} \left( \sum_{i=0}^{\Delta^\lambda_k} x^{iM^\lambda_k}\right).
\]
Then the number of orbits of the centralizer of $T$ of dimension $n$ is given
by the coefficient of $x^n$ in the generating function
\[
\prod_{\lambda\in\mathrm{Spec}(T)} f_\lambda(x).
\]  
\end{theorem}

\begin{example}
Let $T$ be nilpotent, with Jordan blocks of sizes 1, 3, and 5 only, as in Example~\ref{bigex}, and assume the multiplicity of the Jordan blocks is 1, 1, and 1 respectively.  The block increment sequence is then $(1, 2, 2)$,  and
the multiplicity sequence is $(1,1,1)$, with tails $(3, 2, 1)$.  The generating function is then 
\begin{align*}
f(x) &= (1+x^3)(1 + x^2 + x^4)(1 + x + x^2)\\
&= 1 + x + 2x^2 + 2x^3 + 3x^4 + 3x^5 + 2x^6 + 2x^7 + x^8+x^9,
\end{align*}
yielding a total of $f(1) = 18$ orbits, occupying all dimensions from $0$ to $9$.
\end{example}

\section{Acknowledgments}

The idea to study centralizer orbits was given to us by Roger Howe during the PCMI
workshop on Lie Theory in the summer of 1998, during which much of this work was completed. We
apologize for the delay in publication.  We thank Robert Bryant, Chris Douglas, Mike
Hill, Marcus Hum, and especially John Labute for helpful conversations. We thank the referee for
several improvements to the paper.


\end{document}